\documentclass{article}

\usepackage{amssymb, amsmath, amsfonts, 
amsthm, 
epsfig}
\usepackage{epsfig,latexsym}
\usepackage{calc}
\usepackage{array}
\usepackage{graphicx}
\usepackage{subfigure}
\usepackage[indent,bf]{caption}
\usepackage{rotating}
\usepackage{setspace}
\usepackage{longtable}
\usepackage{rotating}
\usepackage{multicol, paralist}
\usepackage{pgf,tikz}
\usepackage{verbatim}
\usepackage{url}
\usepackage{fancyhdr}
\usetikzlibrary{arrows}
\usepackage{hyperref}

\newfont{\blb}{msbm10 scaled\magstep1}
\newtheorem{lem}{Lemma}[section]

\newtheorem{thm}[lem]{Theorem}

\newtheorem{prop}[lem]{Proposition}
\newtheorem{ques}[lem]{Question}

\newtheorem{conj}[lem]{Conjecture}

\title{Hamiltonian chordal graphs are not cycle extendable}

\author{Manuel Lafond\footnotemark[1]
\and 
Ben Seamone\footnotemark[2]\ \footnotemark[3]
}

\begin{document}
\maketitle

\renewcommand{\thefootnote}{\fnsymbol{footnote}}

\footnotetext[1]{Department d'Informatique et de recherche op\'erationnelle, Universit\'e de Montr\'eal, Montreal, QC, Canada, \url{lafonman@iro.umontreal.ca}}
\footnotetext[2]{Department d'Informatique et de recherche op\'erationnelle, Universit\'e de Montr\'eal, Montreal, QC, Canada, \url{seamone@iro.umontreal.ca}}
\footnotetext[3]{This work was supported by the Natural Sciences and Engineering Research Council of Canada.}

\renewcommand{\thefootnote}{\arabic{footnote}}

\begin{abstract}
In 1990, Hendry conjectured that every Hamiltonian chordal graph is cycle extendable; that is, the vertices of any non-Hamiltonian cycle are contained in a cycle of length one greater.  We disprove this conjecture by constructing counterexamples on $n$ vertices for any $n \geq 15$.  Furthermore, we show that there exist counterexamples where
the ratio of the length of a non-extendable cycle to the total number of vertices can be made arbitrarily small.  We then consider cycle extendability in Hamiltonian chordal graphs where certain induced subgraphs are forbidden, notably $P_n$ and the bull.
\end{abstract}

%


\pagestyle{myheadings}
\thispagestyle{plain}
\markboth{M. Lafond and B. Seamone}{Hamiltonian chordal graphs are not cycle extendable}

\section{Introduction}

All graphs considered here are simple, finite, and undirected.  A graph is {\em Hamiltonian} if it has a cycle containing all vertices; such a cycle is a {\em Hamiltonian cycle}. 
A graph $G$ on $n$ vertices is {\em pancyclic} if $G$ contains a cycle of length $m$ for every integer $3 \leq m \leq n$.  
Let $C$ and $C'$ be cycles in $G$ of length $m$ and $m+1$, respectively, such that $V(C') \setminus V(C) = \{v\}$.  We say that $C'$ is an {\em extension} of $C$ and that $C$ is {\em extendable} (or, $C$ {\em extends through} $v$ to $C'$).  If every non-Hamiltonian cycle of $G$ is extendable then $G$ is {\em cycle extendable}.  If, in addition, every vertex of $G$ is contained in a triangle, then $G$ is {\em fully cycle extendable}.  The study of pancyclic graphs was initiated by Bondy \cite{B71}, who recognized that most of the sufficient conditions for Hamiltonicity known at the time in fact implied a more complex cycle structure.  Hendry \cite{H90} introduced the concept of cycle extendability, and proved that many known sufficient conditions for a graph to be pancyclic in fact were sufficient for a graph to be (fully) cycle extendable.

Given a graph $G$ and a set of vertices $U \subseteq V(G)$, we denote by $G[U]$ the subgraph obtained by deleting from $G$ all vertices except those in $U$; $G[U]$ is the subgraph {\em induced} by $U$, and a subgraph of $G$ is an {\em induced subgraph} if it is induced by some $U \subseteq V(G)$.
A graph is {\em chordal} if it contains no induced cycles of length $4$ or greater.  It is not hard to show that every Hamiltonian chordal graph is pancyclic (see Proposition \ref{reduce}), however the question of whether not every Hamiltonian chordal graph is cycle extendable has remained open since 1990:

\begin{conj}[Hendry's Conjecture]\cite{H90}
If $G$ is a Hamiltonian chordal graph, then $G$ is fully cycle extendable.
\end{conj}

In this paper, we settle Hendry's Conjecture in the negative.
In Section \ref{counter}, we show that (a) for any $n \geq 15$ there exists a counterexample to Hendry's Conjecture on $n$ vertices and (b) for every real number $\alpha > 0$ there exists a counterexample $G$ with a non-extendable cycle $C$ such that $|V(C)| < \alpha|V(G)|$.
The question then remains: for which subclasses of the class of chordal graphs is Hendry's Conjecture true?  In Section \ref{subclasses}, we verify the conjecture for some particular chordal graph classes based on forbidden induced subgraphs, and suggest some avenues for further research in Section \ref{open}.

\section{Counterexamples to Hendry's Conjecture}\label{counter}

We continue with some necessary definitions and properties of chordal graphs. 
A set of vertices $X \subseteq V(G)$ which induces a complete subgraph of $G$ is a {\em clique}.
The neighbourhood of a vertex $v \in V(G)$ is the set of vertices to which $v$ is adjacent, which is denoted $N_G(v)$ (or $N(v)$ if the graph in question is clear from context).
A vertex $v \in V(G)$ is called {\em simplicial} if $N_G(v)$ is a clique.  
A {\em perfect elimination ordering} of a graph $G$ is an ordering of $V(G)$, say $v_1 \prec v_2 \prec v_3 \prec \ldots \prec v_n$, such that $v_i$ is simplicial in the graph $G[\{v_i, \ldots, v_n\}]$
for all $i \in \{1,2,\ldots,n\}$.
A {\em vertex cut} of a graph $G$ is a set $X \subset V(G)$ such that $G - X$ is a disconnected graph.
Let $G$ and $H$ be two graphs for which $V(G) \cap V(H)$ is a clique.  We call the graph with vertex set $V(G) \cup V(H)$ and edge set $E(G) \cup E(H)$ the {\em clique sum} of $G$ and $H$; this is also called a {\em clique pasting} of $G$ and $H$.

For a graph $G$, the following statements are well known to be equivalent:
	\begin{compactitem}
	\item $G$ is chordal.
	\item Every minimal vertex cut of every induced subgraph of $G$ is a clique \cite{D61}.
	\item $G$ admits a perfect elimination ordering \cite{D61, FG65}.
	\end{compactitem}
It easily follows that $G$ is chordal if and only if $G$ can obtained from two chordal graphs $G_1$ and $G_2$, with $V(G_1) \subsetneq V(G)$ and $V(G_2) \subsetneq V(G)$, via clique pasting.

We build our counterexamples to Hendry's Conjecture using the graph $H$ given in Figure \ref{base}.

\begin{figure}[h!]
\begin{center}
\scalebox{0.55}{
\begin{tikzpicture}
\clip(3.2,-8) rectangle (15.21,3.8);
\draw (4,-2)-- (12.5,1.5);
\draw (12.5,-5.5)-- (9,-7);
\draw (9,3)-- (5.5,1.5);
\draw (4,-2)-- (14,-2);
\draw (4,-2)-- (12.5,-5.5);
\draw (4,-2)-- (9,-7);
\draw (4,-2)-- (9,-3);
\draw (4,-2)-- (9,3);
\draw [line width=6pt] (4,-2)-- (5.5,1.5);
\draw (4,-2)-- (5.5,-5.5);
\draw (4,-2)-- (7.25,-6.25);
\draw [line width=6pt] (12.5,1.5)-- (14,-2);
\draw (12.5,1.5)-- (9,-3);
\draw (12.5,1.5)-- (9,3);
\draw [line width=6pt] (14,-2)-- (12.5,-5.5);
\draw (14,-2)-- (9,-7);
\draw (14,-2)-- (9,-3);
\draw (14,-2)-- (9,3);
\draw (14,-2)-- (5.5,1.5);
\draw (5.5,-5.5)-- (14,-2);
\draw (7.25,-6.25)-- (14,-2);
\draw [line width=6pt] (9,-7)-- (9,-3);
\draw (7.25,-6.25)-- (9,-7);
\draw [line width=6pt] (9,-3)-- (9,3);
\draw (9,-3)-- (5.5,1.5);
\draw (9,-3)-- (7.25,-6.25);
\draw (9,-3)-- (5.5,-5.5);
\draw (5.5,1.5)-- (5.5,-5.5);
\draw (5.5,-5.5)-- (7.25,-6.25);
\fill [color=black] (4,-2) circle (6.0pt);
\draw[color=black] (3.5,-2) node {\LARGE $a$};
\fill [color=black] (12.5,1.5) circle (6.0pt);
\draw[color=black] (12.8,2) node {\LARGE $d$};
\fill [color=black] (14,-2) circle (6.0pt);
\draw[color=black] (14.5,-2) node {\LARGE $e$};
\fill [color=black] (9,-7) circle (6.0pt);
\draw[color=black] (9,-7.5) node {\LARGE $g$};
\fill [color=black] (9,-3) circle (6.0pt);
\draw[color=black] (9.5,-3.3) node {\LARGE $h$};
\fill [color=black] (9,3) circle (6.0pt);
\draw[color=black] (9,3.5) node {\LARGE $c$};
\fill [color=black] (5.5,1.5) circle (6.0pt);
\draw[color=black] (5.2,2) node {\LARGE $b$};
\fill [color=black] (12.5,-5.5) circle (6.0pt);
\draw[color=black] (12.8,-6) node {\LARGE $f$};
\fill [color=black] (5.5,-5.5) circle (6.0pt);
\draw[color=black] (5.2,-6) node {\LARGE $z_1$};
\fill [color=black] (7.25,-6.25) circle (6.0pt);
\draw[color=black] (7.25,-6.9) node {\LARGE $z_2$};
\end{tikzpicture}
}
\caption{The base graph $H$}\label{base}
\end{center}
\end{figure}
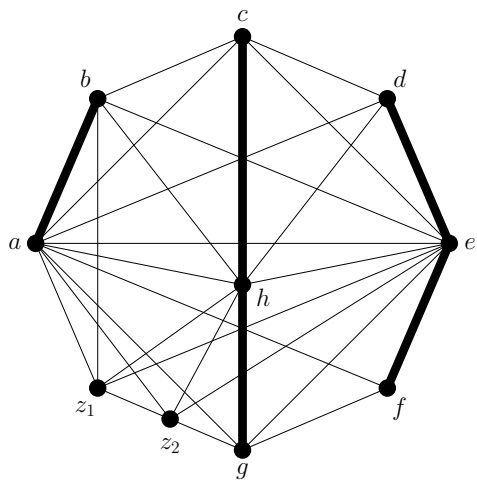

Since $f \prec g \prec z_2 \prec z_1 \prec b \prec a \prec c \prec d \prec e \prec h$
 is a perfect elimination ordering of $H$, $H$ is a chordal graph.
Call the edges $ab, de, ef, ch,$ and $gh$ {\em heavy}; these edges are highlighted in Figure \ref{base}.  We define the following two cycles of $H$:
\begin{align*}
C^* &= abz_1z_2ghcdefa \\
C &= abchgfeda
\end{align*}

Note that $C$ and $C^*$ each contain every heavy edge of $H$.  Furthermore, $C^*$ is a Hamiltonian cycle of $H$ and $C$ spans every vertex of $H$ except $z_1$ and $z_2$.

\begin{lem}\label{heavyextlem}
No extension of $C$ in $H$ contains every heavy edge.
\end{lem}

\begin{proof}
Suppose, to the contrary, such an extension exists.  We may remove from consideration any edge incident to $e$ or $h$ that is not heavy, as well as the edge $z_1z_2$.  The remaining available edges for our desired extension are shown in Figure \ref{heavyextension}.
Since $C$ cannot extend through $z_1$, any extension must contain the edges $az_2$ and $gz_2$.  We may now remove from consideration every other edge incident to $a$ or $g$.  This leaves no remaining edges incident to $f$ to include in an extension of $C$, and hence no such extension exists. \qquad
\end{proof}

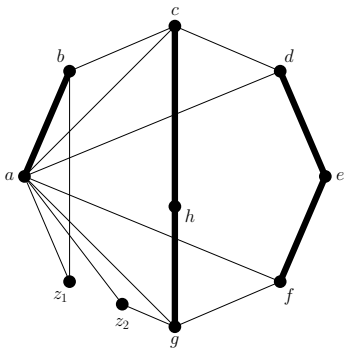
\begin{figure}[h!]
\begin{center}
\scalebox{0.4}{
\begin{tikzpicture}
\clip(3.2,-7.8) rectangle (15.21,3.8);
\draw (4,-2)-- (12.5,1.5);
\draw (12.5,-5.5)-- (9,-7);
\draw (9,3)-- (5.5,1.5);
\draw (4,-2)-- (12.5,-5.5);
\draw (4,-2)-- (9,-7);
\draw (4,-2)-- (9,3);
\draw [line width=6pt] (4,-2)-- (5.5,1.5);
\draw (4,-2)-- (5.5,-5.5);
\draw (4,-2)-- (7.25,-6.25);
\draw [line width=6pt] (12.5,1.5)-- (14,-2);
\draw (12.5,1.5)-- (9,3);
\draw [line width=6pt] (14,-2)-- (12.5,-5.5);
\draw [line width=6pt] (9,-7)-- (9,-3);
\draw (7.25,-6.25)-- (9,-7);
\draw [line width=6pt] (9,-3)-- (9,3);
\draw (5.5,1.5)-- (5.5,-5.5);
\fill [color=black] (4,-2) circle (6.0pt);
\draw[color=black] (3.5,-2) node {\LARGE $a$};
\fill [color=black] (12.5,1.5) circle (6.0pt);
\draw[color=black] (12.8,2) node {\LARGE $d$};
\fill [color=black] (14,-2) circle (6.0pt);
\draw[color=black] (14.5,-2) node {\LARGE $e$};
\fill [color=black] (9,-7) circle (6.0pt);
\draw[color=black] (9,-7.5) node {\LARGE $g$};
\fill [color=black] (9,-3) circle (6.0pt);
\draw[color=black] (9.5,-3.3) node {\LARGE $h$};
\fill [color=black] (9,3) circle (6.0pt);
\draw[color=black] (9,3.5) node {\LARGE $c$};
\fill [color=black] (5.5,1.5) circle (6.0pt);
\draw[color=black] (5.2,2) node {\LARGE $b$};
\fill [color=black] (12.5,-5.5) circle (6.0pt);
\draw[color=black] (12.8,-6) node {\LARGE $f$};
\fill [color=black] (5.5,-5.5) circle (6.0pt);
\draw[color=black] (5.2,-6) node {\LARGE $z_1$};
\fill [color=black] (7.25,-6.25) circle (6.0pt);
\draw[color=black] (7.25,-6.9) node {\LARGE $z_2$};
\end{tikzpicture}
}
\caption{Available edges of $H$ for an extension of $C$ that contains all heavy edges}\label{heavyextension}
\end{center}
\end{figure}

\begin{thm}\label{sizecounter}
For any $n \geq 15$, there exists a counterexample to Hendry's Conjecture on $n$ vertices.
\end{thm}

\begin{proof}
Let $G$ be a graph obtained from $H$ by pasting a clique onto each heavy edge of $H$ so that $|V(G)| = n \geq 15$.  Since $G$ is obtained from $H$ and a disjoint set of complete graphs by clique pasting, $G$ is chordal.  Let $D^*$ and $D$ be cycles of $G$ obtained from $C^*$ and $C$, respectively, by replacing each heavy edge $xy$ with a Hamiltonian $xy$-path through the clique which was pasted onto $xy$ to obtain $G$.  We see that $D^*$ is a Hamiltonian cycle of $G$ and that $D$ is a cycle that spans every vertex of $G$ except $z_1$ and $z_2$.  Furthermore, $D$ cannot be extended in $G$, otherwise $C$ could have been extended in $H$ using every heavy edge, a contradiction of Lemma \ref{heavyextlem}. \qquad
\end{proof}

For any fixed counterexample on $n$ vertices constructed in the proof of Theorem \ref{sizecounter}, consider the graph obtained by pasting a clique of size $k$ onto the edge $z_1z_2$.  Such a graph is still Hamiltonian and a cycle $D$ as given in the proof of Theorem \ref{sizecounter} cannot be extended.  Since we have a cycle of length $n-2$ that cannot be extended in a graph on $n+k-2$ vertices, we obtain the following:

\begin{thm}\label{ratiocounter}
For any real number $\alpha>0$, there exists a Hamiltonian chordal graph $G$ with a non-extendable cycle $C$ satisfying $|V(C)|< \alpha|V(G)|$.
\end{thm}

To conclude the section, we note that the construction given in the proof of Theorem \ref{sizecounter} does not necessarily require $5$ cliques to be pasted onto the heavy edges of $H$ -- any set of $5$ Hamiltonian chordal graphs $\{G_1, G_2, G_3, G_4, G_5\}$ will suffice, where the edge of $G_i$ pasted onto a heavy edge of $H$ can be chosen to be any edge from any Hamiltonian cycle in $G_i$.


\section{Hamiltonian chordal graphs which are fully cycle extendable}\label{subclasses}

Even though Hendry's Conjecture is not true in general, it is still interesting to consider sufficient conditions for a chordal graph to be fully cycle extendable.  
A graph is {\em $H$-free} if it contains no induced subgraph isomorphic to $H$, and it is {\em $\mathcal{H}$-free} for a set of graphs $\mathcal{H}$ if it is $H$-free for every $H \in \mathcal{H}$.
The remainder of this paper is concerned with graphs characterized by forbidden induced subgraphs.

Chordal graphs are one obvious example of a graph class characterized by forbidden induced subgraphs; they are by definition $\{C_4, C_5, C_6, \ldots\}$-free.
A {\em strongly chordal graph} is defined to be a chordal graph in which even cycle of length at least $6$ has a chord that connects vertices at an odd distance from one another along the cycle.  
Strongly chordal graphs can also be characterized by forbidden induced subgraphs.
A {\em $k$-sun} is a chordal graph $G$ whose vertices can be partitioned into two sets $X = \{x_1, x_2, \ldots, x_k\}$ and $Y = \{y_1, y_2, \ldots, y_k\}$ such that $x_i$ is adjacent only to $y_i$ and $y_{i+1}$ in $G$ (subscripts taken modulo $k$).  A graph is a {\em sun} if it is a $k$-sun for some $k$.
Farber \cite{F83} showed that a graph is strongly chordal if and only if it is chordal and sun-free.

We now summarize the classes for which Hendry's Conjecture is known to hold.
A Hamiltonian chordal graph is fully cycle extendable if it is also 
\begin{compactitem}
\item planar \cite{J02},
\item a spider intersection graph (the intersection graph of subtrees of a subdivided star) \cite{ABS13},
\item strongly chordal and $(K_{1,4}+e)$-free \cite{AS06}, or
\item strongly chordal and hourglass-free \cite{AS06}\label{strong}. 
\end{compactitem}
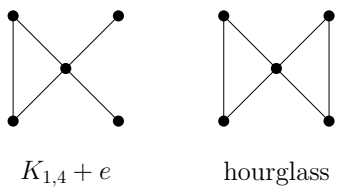
\begin{figure}[h!]
\begin{center}
\scalebox{0.7}{
\begin{tikzpicture}
\clip(-0.5,-1.4) rectangle (6.5,2.5);
\draw (2,2) -- (0,0) -- (0,2) -- (2,0);
\draw (6,2) -- (4,0) -- (4,2) -- (6,0) -- (6,2);

\begin{scriptsize}
\fill [color=black] (0,0) circle (3pt);
\fill [color=black] (2,0) circle (3pt);
\fill [color=black] (2,2) circle (3pt);
\fill [color=black] (0,2) circle (3pt);
\fill [color=black] (1,1) circle (3pt);

\fill [color=black] (4,0) circle (3pt);
\fill [color=black] (6,0) circle (3pt);
\fill [color=black] (6,2) circle (3pt);
\fill [color=black] (4,2) circle (3pt);
\fill [color=black] (5,1) circle (3pt);

\end{scriptsize}
\draw (1, -1) node[anchor=center] {\Large $K_{1,4}+e$};
\draw (5, -1) node[anchor=center] {\Large hourglass};
\end{tikzpicture}
}
\end{center}
\vspace{-0.2in} \caption{$K_{1,4} + e$ and the hourglass}\label{5verts}
\end{figure}
The result on spider intersection graphs generalizes previous results on interval graphs \cite{AS06,CFGJ06} and split graphs \cite{AS06}. 

One can obtain other classes of graphs for which Hendry's Conjecture holds by looking at results on locally connected graphs.  A graph $G$ is {\em locally connected} if $N(v)$ induces a connected subgraph of $G$ for every $v \in V(G)$.

\begin{prop}\label{2conntolocconn}
A connected chordal graph is $2$-connected if and only if it is locally connected.
\end{prop}

\begin{proof}
Chartrand and Pippert \cite{CP74} proved that every connected and  locally connected chordal graph is $2$-connected, so we need only consider the ``only if'' portion of the statement.
Suppose $G$ is a $2$-connected chordal graph and let $x,y \in N(v)$ for some $v$.  Since $G$ is $2$-connected, there exists a cycle through $xv$ that contains $y$; equivalently, there is an $xy$-path that avoids $v$.  Suppose that every such path has vertices not in $N(v)$.  Let $P$ be a shortest such path, and let $Q$ be a segment of $P$ with ends $a,b$ lying in $N(v)$ and all internal vertices not in $N(v)$.  By minimality of $P$, we have that $Q + avb$ is an induced cycle of length at least $4$, a contradiction. \qquad
\end{proof}

A connected and locally connected graph $G$ is known to be cycle extendable if it also satisfies one of the following conditions: 
\begin{compactitem}
\item $\Delta(G) = 5$ and $\delta(G) \geq 3$ \cite{GOPS11},
\item $G$ is almost claw-free \cite{R94}, and hence claw-free (originally shown in \cite{C81, H90}),
\item $G$ is $\{K_{1,4}, K_{1,4}+e\}$-free with $\delta(G) \geq 3$ \cite{LW09}.
\end{compactitem}
Proposition \ref{2conntolocconn} implies that such graphs are fully cycle extendable if they are chordal and $2$-connected.

In the remainder of this section, we consider more forbidden induced subgraphs which guarantee that a Hamiltonian chordal graph is fully cycle extendable.  
In particular, we show that any Hamiltonian chordal graph which also falls into one of the classes of graphs listed below is fully cycle extendable (see Figure \ref{more5verts} for the graphs in question):
	\begin{compactitem}
	\item $P_5$-free;
	\item $\{\text{bull}, K_{1,5}\}$-free;
	\item $\{\text{bull}, K_2 \vee P_5\}$-free;
	\end{compactitem}
Recall that the {\em join} of two graphs $G$ and $H$, denoted $G \vee H$, is the graph with vertex set $V(G) \cup V(H)$ and edge set $E(G) \cup E(H) \cup \{gh \,:\, g \in V(G), h \in V(H)\}$. \\
	
\begin{figure}[h!]
\begin{center}
\scalebox{0.7}{
\begin{tikzpicture}
\clip(-0.5,-1.4) rectangle (20.5,2.2);

\draw (0,1) -- (4,1);

\draw (6,2) -- (7, 1.3) -- (9, 1.3) -- (10,2);
\draw (7, 1.3) -- (8,0) -- (9, 1.3);

\draw (11,0) -- (13,2);
\draw (12,0) -- (13,2);
\draw (13,0) -- (13,2);
\draw (14,0) -- (13,2);
\draw (15,0) -- (13,2);

\draw (18,0) -- (16,1) -- (18,2);
\draw (18,0) -- (17,1) -- (18,2);
\draw (18,0) -- (18,1) -- (18,2);
\draw (18,0) -- (19,1) -- (18,2);
\draw (18,0) -- (20,1) -- (18,2) -- (18,0);
\draw (16,1) -- (20,1);
\draw (18,2) to [out=300, in=60] (18,0);

\begin{scriptsize}

\fill [color=black] (0,1) circle (3pt);
\fill [color=black] (1,1) circle (3pt);
\fill [color=black] (2,1) circle (3pt);
\fill [color=black] (3,1) circle (3pt);
\fill [color=black] (4,1) circle (3pt);

\fill [color=black] (6,2) circle (3pt);
\fill [color=black] (7,1.3) circle (3pt);
\fill [color=black] (9,1.3) circle (3pt);
\fill [color=black] (10,2) circle (3pt);
\fill [color=black] (8,0) circle (3pt);

\fill [color=black] (11,0) circle (3pt);
\fill [color=black] (12,0) circle (3pt);
\fill [color=black] (13,0) circle (3pt);
\fill [color=black] (14,0) circle (3pt);
\fill [color=black] (15,0) circle (3pt);
\fill [color=black] (13,0) circle (3pt);

\fill [color=black] (16,1) circle (3pt);
\fill [color=black] (17,1) circle (3pt);
\fill [color=black] (18,1) circle (3pt);
\fill [color=black] (19,1) circle (3pt);
\fill [color=black] (20,1) circle (3pt);
\fill [color=black] (18,2) circle (3pt);
\fill [color=black] (18,0) circle (3pt);

\end{scriptsize}

\draw (2, -1) node[anchor=center] {\Large $P_5$};
\draw (8, -1) node[anchor=center] {\Large bull};
\draw (13, -1) node[anchor=center] {\Large $K_{1,5}$};
\draw (18, -1) node[anchor=center] {\Large $K_2 \vee P_5$};

\end{tikzpicture}
}

\end{center}
\caption{Forbidden induced subgraphs considered in Sections \ref{otherext} and \ref{bullfree}}\label{more5verts}
\end{figure}
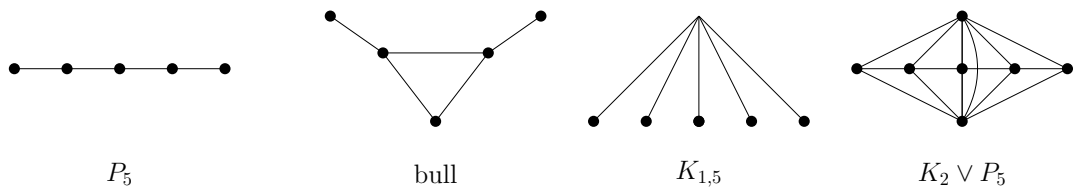

\subsection{$P_5$-free chordal graphs}\label{otherext}

We begin with some technical results on cycles in chordal graphs, particularly as they relate to vertex cuts or cutsets.
For two vertices $u,v \in V(G)$, a {\em $uv$-separator} is a set of vertices $X \subset V(G)$ such that $u$ and $v$ lie in different connected components of $G - X$.
A {\em minimal $uv$-separator} is a $uv$-separator which has no proper subset which is also a $uv$-separator.
A {\em separator} is a $uv$-separator for some $u,v \in V(G)$, and
a {\em minimal separator} is a minimal $uv$-separator for some $u,v \in V(G)$.
A vertex $v$ is {\em complete} to a set $X \subseteq V(G) \setminus \{v\}$ if $v$ is adjacent to every vertex in $X$, and $v$ is {\em anticomplete} to $X$ if it is non-adjacent to every vertex in $X$.  For two disjoint subsets of $V(G)$, say $X$ and $Y$, we say $X$ is complete (anticomplete) to $Y$ if every vertex in $X$ is complete (anticomplete) to $Y$.

As stated earlier, it was shown by Dirac \cite{D61} that if $G$ is a chordal graph and $X \subseteq V(G)$ is a minimal separator (i.e. a minimal vertex cut), then $X$ is a clique.  Furthermore, since every chordal graph has a perfect elimination ordering, it follows that every chordal graph has a simplicial vertex, and if $G$ is not a clique then $G$ contains at least two nonadjacent simplicial vertices.  The following, more general statement, easily follows:

\begin{prop}\label{separator}
If $G$ is a chordal graph and $X$ is a minimal separator, then every connected component of $G - X$ contains a simplicial vertex of $G$.
\end{prop}

As a corollary, we obtain a more general result on (not necessarily minimal) clique separators:

\begin{prop}\label{clique}
If $X$ is a clique separator of a chordal graph $G$, then each connected component of $G - X$ contains a simplicial vertex of $G$.
\end{prop}

\begin{proof}
Suppose that $Q$ is a connected component of $G - X$.  Let $H$ be the graph obtained from $G[X \cup V(Q)]$ by adding a vertex $v$ and all edges from $v$ to $X$.  For any $u \in V(Q)$, $X$ is a $uv$-separator in $H$, and hence $X$ contains a minimal $uv$-separator, $Y$.  Since $v$ is complete to $X$, it is complete to $X \setminus Y$, and hence $H - Y$ has $Q$ as a connected component (and $H[\{v\} \cup (X \setminus Y)]$ as the other).  By Proposition \ref{separator}, $Q$ contains a simplicial vertex in $H$, say $x$.  However, $H[X \cup V(Q)] = G[X \cup V(Q)]$ and all neighbours of $x$ in $G$ lie in $V(Q)$ or $X$, and thus $x$ is a simplicial vertex in $G$ as well. \qquad
\end{proof}

The following simple proposition implies that every Hamiltonian chordal graph is pancyclic and thus ``cycle reducible''; it is this fact that originally inspired Hendry's Conjecture. 

\begin{prop}\label{reduce}
Let $G$ be a Hamiltonian chordal graph.  If $v$ is a simplicial vertex of $G$, then $G-v$ is Hamiltonian.
\end{prop}

We may now show that Hendry's Conjecture holds for $P_5$-free graphs.

\begin{thm}\label{P5}
Every $P_5$-free Hamiltonian chordal graph is fully cycle extendable.
\end{thm}

\begin{proof}
Let $G$ be a $P_5$-free Hamiltonian chordal graph, and let $C$ be a non-Hamiltonian cycle in $G$.
We prove the statement by induction on $|V(G)|$.  The statement can be easily checked for sufficiently small graphs, say for $|V(G)| \leq 5$.  
Assume that $|V(G)| \geq 6$ and that, for any $P_5$-free Hamiltonian chordal graph $G'$ with $|V(G')| < |V(G)|$, any non-Hamiltonian cycle $C'$ extends in $G'$.
Let $Q$ be a connected component of $G - V(C)$ and let $X \subseteq V(C)$ be those vertices of $C$ with a neighbour in $Q$ (note that $X$ necessarily contains at least $2$ vertices).

If $X$ is a clique, then $Q$ contains a simplicial vertex by Proposition \ref{clique}, say $v$.
By Proposition \ref{reduce}, $G - v$ is Hamiltonian.
Clearly $C$ is a cycle in $G-v$.  If $C$ is Hamiltonian in $G-v$, then a Hamiltonian cycle in $G$ is an extension of $C$ in $G$.  If $C$ is not Hamiltonian in $G-v$, then $C$ extends in $G-v$ by induction and hence also in $G$.

Suppose, then, that there exist $x,y \in X$ which are nonadjacent; let $a$ and $b$ denote the neighbours of $x$ on $C$ and let $c$ and $d$ denote the neighbours of $y$ on $C$.
We will show that $C$ contains an edge whose ends share a neighbour outside of $C$, and hence $C$ easily extends.
  
First, suppose that each of $x$ and $y$ have a cycle neighbour which does not lie in $X$, say $a$ and $c$ (we do not assume that these vertices are distinct).  Since $xy \notin E(G)$, a shortest $xy$-path in $G[V(Q) \cup \{x,y\}]$ has length at least $2$.  Let $P$ be such a path, and consider the subgraph of $G$ induced by $V(P) \cup \{a,c\}$.  If $a = c$, then $ax + P + ya$ is a cycle of length at least $4$ in $G$.  However, since $a$ has no neighbours in $Q$ and $P$ is minimal, this cycle is chordless, a contradiction.  Thus, we assume that $a \neq c$.  The only possible edges induced by the vertices of the path $ax + P + yc$ (other than those in the path themselves) are $ay, ac,$ and $xc$.  Since $G$ is $P_5$ free, at least one such edge must be present, however any combination creates a chordless cycle of length at least $4$, a contradiction.

We now may assume that one of $x$ or $y$ has both of its cycle neighbors in $X$.  Without loss of generality, we may assume that $a \in X$.  Let $Y_a = V(Q) \cap N_G(a)$ and $Y_x = V(Q) \cap N_G(x)$.  Suppose that $Y_a \cap Y_x = \emptyset$.  Out of all paths connecting some vertex in $Y_a$ to some vertex in $Y_x$, let $P$ be one of minimum length; say that $P$ joins $s \in Y_a$ and $t \in Y_x$.  By construction, no internal vertex of $P$ is adjacent to either $a$ or $x$ in $G$, and hence $P + txas$ is an induced cycle of length at least $4$, a contradiction.  It follows that $Y_a \cap Y_x$ is nonempty.  If $t \in Y_a \cap Y_x$, then $C' = C - ax + atx$ is our desired extension of $C$ in $G$. \qquad
\end{proof}

\subsection{Bull-free chordal graphs}\label{bullfree}

Bull-free chordal graphs are of interest to us for two reasons.  The first is that bull-free graphs are historically tied to the study of perfect graphs, of which chordal graphs form a well-known subclass.  The second is that there is some evidence that Hendry's Conjecture may hold for strongly chordal graphs (see \cite{ABS13, AS06,CFGJ06}), and bull-free chordal graphs are strongly chordal.  To see this, recall that a graph is strongly chordal if and only if it is chordal and sun-free.  Since every sun contains a bull, any bull-free chordal graph is sun-free, and hence strongly chordal.
%


While we cannot yet show that bull-free Hamiltonian chordal graphs are fully cycle extendable, we can show that $\{\text{bull}, X\}$-free Hamiltonian chordal graphs are 
fully cycle extendable for two reasonably large subgraphs $X$.  To do this, we need the following simple observation:

\begin{lem}\label{common}
If $C$ is a cycle in a chordal graph and $uv \in E(C)$, then $u$ and $v$ share a common neighbour on $C$.
\end{lem}

\begin{proof}
There is a shortest cycle through $uv$ in $G[V(C)]$.  Since $G[V(C)]$ is chordal, this cycle must have length $3$; the vertex in this $3$-cycle distinct from $u$ and $v$ is the desired common neighbour. \qquad
\end{proof}

\begin{thm}\label{bullplus}
If $G$ is a $\{\text{bull}, K_{1,5}\}$-free or $\{\text{bull},  K_2 \vee P_5\}$-free Hamiltonian chordal graph, then $G$ is fully cycle extendable.
\end{thm}

\begin{proof}
Let $X$ be either $K_{1,5}$ or $K_2 \vee P_5$.
Let $G$ be a minimal  $\{\text{bull},  X\}$-free Hamiltonian chordal graph which is not cycle extendable, and let $C$ be a non-Hamiltonian cycle which is not extendable.  By only making use of the fact that $G$ is Hamiltonian, bull-free, and chordal, we will show that $G$ necessarily contains both $K_{1,5}$ and $K_2 \vee P_5$, a contradiction.

Consider the vertices of $C$ in some cyclic order.  
For a vertex $a \in V(C)$, we denote by $a^-$ and $a^+$ the vertices immediately preceding and succeeding $a$ along $C$, respectively.
For two vertices $a,b \in V(C)$, let $C[a,b]$ denote the segment of $C$ from $a$ to $b$ with respect to the cyclic ordering (that is, containing $a^+$ and $b^-$).

Let $C^*$ be a Hamiltonian cycle of $G$.  There must be a segment of $C^*$ with at least $3$ vertices whose ends lie on $C$ and whose internal vertices are disjoint from $V(C)$.  Choose $Z = uz_1\cdots z_kv$ to be a shortest such segment; for notation purposes let $z_0 = u$ and $z_{k+1} = v$, and let $\hat{Z}$ denote the internal vertices of $Z$.
We now argue the presence or absence of edges in the induced subgraph of $G$ having vertex set $V(C) \cup V(Z)$.  Figure \ref{edgesin} displays the edges which we argue are present.

We first argue that $uv \notin E(G)$.  Suppose, to the contrary, that $uv \in E(G)$.  This implies that  $uv + Z$ is a cycle in $G$ and so, by Lemma \ref{common}, $u$ and $v$ share a common neighbour on $Z$, say $z$.
If $\hat{Z} = \{z\}$, then $G - z$ is Hamiltonian (we replace the segment $uzv$ in $C^*$ with $uv$) and we have that $C$ is extendable in $G - z$ by the minimality of $G$.  If $|\hat{Z}| \geq 2$, then consider the graph $G' = G - \left(\hat{Z} \setminus \{z\}\right)$.  This graph is still Hamiltonian (we replace $Z$ in $C^*$ with $uzv$), so $C$ extends in $G'$ by the minimality of $G$, and hence also extends in $G$.  Thus, $uv \notin E(G)$.

Now, we note that $u^-$ and $u^+$ must be non-adjacent to $z_1$ and that $v^-$ and $v^+$ must be non-adjacent to $z_k$, otherwise $C$ is extendable.
We also may deduce that $u$ is not adjacent to any of $z_2, \ldots, z_k$, $v$ is not adjacent to any of $z_1, \ldots, z_{k-1}$, and no edge connects two vertices of $\hat{Z}$ except those edges of $Z$ itself.  If any of these were not the case, then there are vertices of $Z$ that may be deleted that maintain Hamiltonicity, and $C$ then extends by the minimality argument.

Among all vertices on $C[u,v]$, let $x$ be the neighbour of $u$ that is closest to $v$.  Similarly, let $y$ denote the vertex on $C[v,u]$ that is the neighbour of $u$ closest to $v$.  By Lemma \ref{common}, $u$ and $x$ must share a common neighbour on $C[x,y] \cup yux$.  Since $u$ is adjacent only to $x$ and $y$ on this cycle, this common neighbour must by $y$ and so $xy \in E(G)$.

We now show that $z_ix, z_iy \in E(G)$ and $z_ix^-, z_ix^+, z_iy^-, z_iy^+ \notin E(G)$ for every $i = 1, \ldots, k$.
Consider the cycle $Z \cup ux \cup C[x,v]$.  Note that the only neighbours of $u$ on this cycle are $x$ and $z_1$.  By Lemma \ref{common}, $u$ and $x$ must share a common neighbour on this cycle, and hence $z_1x \in E(G)$.  The non-extendability of $C$ implies that $x^-z_1,x^+z_1 \notin E(G)$.  Note that we may now deduce that $x$ is distinct from both $u^+$ and $v^-$.
Now consider the subgraph induced by $\{u, x, x^+, z_1, z_2\}$.  The triangle $uxz_1$ together with the edges $z_1z_2$ and $xx^+$ form a bull.  We have already deduced that $ux^+, uz_2, x^+z_1 \notin E(G)$; since $G$ is chordal and bull-free, we must have that $xz_2 \in E(G)$.  It follows that $x^-z_2, x^+z_2 \notin E(G)$.  We then iterate this argument with the triangle $z_iz_{i+1}x$ and the edges $xx^+$ and $z_{i+1}z_{i+2}$ to argue that $x$ is complete to $\hat{Z}$ and that $x^-$ and $x^+$ are anticomplete to $\hat{Z}$.  Now, if $x^+ = v^-$, then $xv \in E(G)$, since $\{x,v^-, v, z_k\}$ cannot induce a $4$-cycle and $v^-z_k \notin E(G)$.  If $x^+ \neq v^-$, then since $\{x, x^+, z_{k-1}, z_k, v\}$ cannot induce a bull, we must also have $xv \in E(G)$.
An identical argument gives that $y$ is complete to $\hat{Z} \cup \{v\}$, that $y$ is distinct from $u^-$ and $v^+$, and that $y^-$ and $y^+$ are anticomplete to $\hat{Z}$.
Since $G$ is chordal and $u^-z_1, u^+z_1, v^-z_k, v^+z_k \notin E(G)$, it is also easy to deduce that $\{u^-,u^+, v^-, v^+\}$ is anticomplete to $\hat{Z}$.

We now show that $\{u^-, u, u^+\}$ is anticomplete to $\{v^-, v, v^+\}$, that the vertices in $\{u^+, x^+, v^+, y^+\}$ are pairwise non-adjacent, and that the vertices in $\{u^-, x^-, v^-, y^-\}$, are pairwise non-adjacent.
If $uv^- \in E(G)$, then $Z \cup vv^-u$ would be a chordless cycle of length at least $4$, a contradiction.  An identical argument holds for $uv^+, vu^-, $ and $vu^+$.
Now, if any of $u^+x^+, x^+v^+, v^+y^+$ or $y^+u^+$ are edges of $G$, then $C$ easily extends (for instance, $u^+x^+ \in E(G)$ gives the extension $C - uu^+ - xx^+ + u^+x^+ + uz_1x$).  Hence, $u^+x^+, x^+v^+, v^+y^+,y^+u^+ \notin E(G)$.  If $x^+y^+ \in E(G)$, then $C$ extends to $C - xx^+ - yy^+ + x^+y^+ + xz_1y$, a contradiction.  Also, if any edge of $G$ has one end in $\{u^-, u^+\}$ and the other in $\{v^-,v^+\}$, then $G$ contains an induced cycle of length at least $5$, a contradiction.  Thus, $\{u^-, u, u^+\}$ is anticomplete to $\{v^-, v, v^+\}$ and the vertices in $\{u^+, x^+, v^+, y^+\}$ are pairwise non-adjacent.  An identical argument gives that the vertices of $\{u^-, x^-, v^-, y^-\}$ are pairwise non-adjacent.

\begin{figure}[h!]
\begin{center}
\scalebox{0.4}{
\definecolor{qqqqff}{rgb}{0,0,1}
\definecolor{ffqqqq}{rgb}{1,0,0}
\begin{tikzpicture}
\clip(-6.22,-9.1) rectangle (18,13.16);

\draw [line width=5pt,color=black] (-4,2)-- (2,2);
\draw [line width=5pt,color= black] (2,2)-- (4.5,2);
\draw [line width=2pt,color= black, dotted] (4.5,2)-- (10,2);
\draw [line width=5pt,color= black] (10,2)-- (16,2);

\draw [line width=5pt,color= black] (-4,2)-- (-3,6);
\draw [line width=5pt,color= black] (-4,2)-- (-3,-2);
\draw [line width=2pt,color= black, dotted] (-3,6) to [out=60, in=210] (2,11);
\draw [line width=5pt,color= black] (2,11)-- (6,12);
\draw [line width=5pt,color= black] (6,12)-- (10,11);
\draw [line width=2pt,color= black, dotted] (10,11) to [out=330, in=120] (15,6);
\draw [line width=5pt,color= black] (15,6)-- (16,2);
\draw [line width=5pt,color= black] (16,2)-- (15,-2);
\draw [line width=2pt,color= black, dotted] (15,-2) to [out=240, in= 30](10,-7);
\draw [line width=5pt,color= black] (10,-7)-- (6,-8);
\draw [line width=5pt,color= black] (6,-8)-- (2,-7);
\draw [line width=2pt,color= black, dotted] (2,-7) to [out=150, in=300] (-3,-2);


\draw [line width=2pt] (6,12)-- (6,-8);
\draw [line width=2pt] (2,2)-- (6,12);
\draw [line width=2pt] (6,-8)-- (2,2);
\draw [line width=2pt] (16,2)-- (6,12);
\draw [line width=2pt] (15,6)-- (6,12);
\draw [line width=2pt] (16,2)-- (6,-8);
\draw [line width=2pt] (-4,2)-- (6,12);
\draw [line width=2pt] (6,-8)-- (-4,2);
\draw [line width=2pt] (6,12)-- (-3,-2);
\draw [line width=2pt] (6,12)-- (-3,6);
\draw [line width=2pt] (6,-8)-- (-3,6);
\draw [line width=2pt] (6,-8)-- (-3,-2);
\draw [line width=2pt] (15,-2)-- (6,12);
\draw [line width=2pt] (6,-8)-- (15,6);
\draw [line width=2pt] (10,2)-- (6,12);
\draw [line width=2pt] (10,2)-- (6,-8);
\draw [line width=2pt] (6,12)-- (4.5,2);
\draw [line width=2pt] (6,-8)-- (4.5,2);

\begin{Huge}

\fill [color=black] (-4,2) circle (6pt);
\draw[color=black] (-4.8,2) node[anchor=center] {$u$};
\fill [color=black] (-3,6) circle (6pt);
\draw[color=black] (-3.7,6.7) node[anchor=center] {$u^+$};
\fill [color=black] (-3,-2) circle (6pt);
\draw[color=black] (-3.7,-2.5) node[anchor=center] {$u^-$};
\fill [color=black] (16,2) circle (6pt);
\draw[color=black] (16.8,2) node[anchor=center] {$v$};
\fill [color=black] (6,12) circle (6pt);
\draw[color=black] (6.27,12.58) node[anchor=center] {$x$};
\fill [color=black] (6,-8) circle (6pt);
\draw[color=black] (6,-8.75) node[anchor=center] {$y$};
\fill [color=black] (2,11) circle (6pt);
\draw[color=black] (1.89,11.9) node[anchor=center] {$x^-$};
\fill [color=black] (10,11) circle (6pt);
\draw[color=black] (10.55,11.9) node[anchor=center] {$x^+$};
\fill [color=black] (15,6) circle (6pt);
\draw[color=black] (16,6.59) node[anchor=center] {$v^-$};
\fill [color=black] (15,-2) circle (6pt);
\draw[color=black] (16,-2.01) node[anchor=center] {$v^+$};
\fill [color=black] (10,-7) circle (6pt);
\draw[color=black] (10.5,-7.6) node[anchor=center] {$y^-$};
\fill [color=black] (2,-7) circle (6pt);
\draw[color=black] (1.6,-7.8) node[anchor=center] {$y^+$};

\fill [color=black] (2,2) circle (6pt);
\draw[color=black] (1.5,2.6) node[anchor=center] {$z_1$};
\fill [color=black] (4.5,2) circle (6pt);
\draw[color=black] (4,2.6) node[anchor=center] {$z_2$};
\fill [color=black] (10,2) circle (6pt);
\draw[color=black] (10.4,2.6) node[anchor=center] {$z_k$};
%

\end{Huge}
\end{tikzpicture}
}

\end{center}
\caption{Edges of a subgraph in $G$ as described in the proof of Theorem \ref{bullplus}}\label{edgesin}
\end{figure}
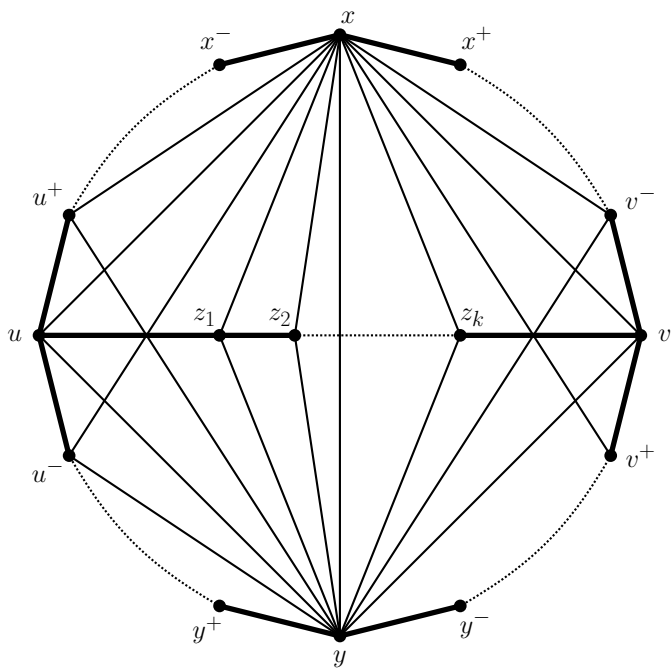

We claim that $x$ is adjacent to $u^-, u^+, v^+$ and that $y$ is adjacent to $u^-, u^+, v^-$.
We begin with $x$ and $u^+$.  If $u^+ = x^-$, then $\{u, u^+, x, z_1\}$ is a $4$-cycle and $u^+x$ is the only possible chord.  Otherwise, the triangle $uxz_1$ and the edges $uu^+$ and $xx^+$ form a bull, and $ux^+, z_1x^+,u^+z_1, u^+x^+$ are non-edges.  Hence, we must have $u^+x \in E(G)$.  An identical argument for the set of vertices $\{u, u^-, x, x^+, z_1\}$ shows that $xu^- \in E(G)$.  
To show that $xv^+ \in E(G)$ requires a little more work.  We first note that we have a bull consisting of the triangle $xz_kv$ and the edges $xx^+$ and $vv^+$.  We have shown that $x^+$, $v^+$, and $z_k$ are mutually non-adjacent, and so one of $x^+v$ or $xv^+$ must be an edge.  If $xv^+ \in E(G)$, then we are done.  Suppose that $x^+v \in E(G)$.  We now consider the bull with triangle $xx^+v$ and pendant edges $ux$ and $vv^+$.  The only possible edge which has not yet been ruled out is $xv^+$, and so it must be an edge in this case as well.
Symmetric arguments gives that $y$ is adjacent to $u^-, u^+, v^-$.

Consider the bull with triangle $z_1xy$ and pendant edges $xv^+$ and $yv^-$.  The edges $v^-z_1$ and $v^+z_1$ are forbidden, and so one of the remaining three edges ($xv^-, yv^+, v^-v^+$) must be present.  Regardless of the presence of $v^-v^+$, one of $xv^-, yv^+$ must be an edge since $G$ is chordal.  We may say without loss of generality (by symmetry) that $xv^- \in E(G)$.
Finally, we consider the bull with triangle $z_1xy$ and pendant edges $xx^-$ and $yy^-$.  Since $x^-$, $y^-$, and $z_1$ are pairwise non-adjacent, we must have that one of $xy^-$ or $yx^-$ is an edge of $G$.

We now see that the following subgraphs are induced in $G$:
	\begin{compactitem}
	\item $K_{1,5}$: $\{x,u^-,x^-,v^-,y^-,z_1\}$ if $xy^- \in E(G)$ and $\{y,u^-,x^-,v^-,y^-,z_1\}$ if $yx^- \in E(G)$;
	\item $K_2 \vee P_5$: $\{x,y,u^-,u,z_1,v,v^-\}$ if $Z$ has one internal vertex and $\{x,y,u^-,u,z_1,z_2,z_3\}$ otherwise (recall that $v = z_{k+1}$).
	\end{compactitem}
	Hence we have arrived at our desired contradiction. \qquad
\end{proof}
 
\section{Future work}\label{open}

As we have mentioned, many of the classes for which Hendry's Conjecture is known to hold are strongly chordal---interval graphs (shown to be strongly chordal in \cite{BLS99}), strongly chordal graphs which are also either $(K_{1,4}+e)$-free or hourglass-free (stated on page \pageref{strong}), and $\{\text{bull}, K_{1,5}\}$-free and $\{\text{bull},  K_2 \vee P_5\}$-free graphs (Corollary \ref{bullplus}).
Furthermore, no counterexample to Hendry's Conjecture which was constructed in Section \ref{counter} is strongly chordal.  To see this, let $H^{+}$ be the graph obtained from $H$ (Figure \ref{base}) by joining a vertex $x$ to the vertices $g$ and $h$ (see Figure \ref{H+}).
\begin{figure}[h!]
\begin{center}
\scalebox{0.4}{
\begin{tikzpicture}
\clip(3.2,-8) rectangle (15.21,3.8);
\draw (4,-2)-- (12.5,1.5);
\draw [line width = 6pt] (12.5,-5.5)-- (9,-7);
\draw (9,3)-- (5.5,1.5);
\draw (4,-2)-- (14,-2);
\draw [line width = 6pt] (4,-2)-- (12.5,-5.5);
\draw [line width = 6pt] (4,-2)-- (9,-7);
\draw [line width = 6pt] (4,-2)-- (9,-3);
\draw (4,-2)-- (9,3);
\draw [line width = 6pt] (4,-2)-- (5.5,1.5);
\draw (4,-2)-- (5.5,-5.5);
\draw (4,-2)-- (7.25,-6.25);
\draw  (12.5,1.5)-- (14,-2);
\draw (12.5,1.5)-- (9,-3);
\draw (12.5,1.5)-- (9,3);
\draw  (14,-2)-- (12.5,-5.5);
\draw (14,-2)-- (9,-7);
\draw (14,-2)-- (9,-3);
\draw (14,-2)-- (9,3);
\draw (14,-2)-- (5.5,1.5);
\draw (5.5,-5.5)-- (14,-2);
\draw (7.25,-6.25)-- (14,-2);
\draw  [line width = 6pt] (9,-7)-- (9,-3);
\draw (7.25,-6.25)-- (9,-7);
\draw  (9,-3)-- (9,3);
\draw [line width = 6pt] (9,-3)-- (5.5,1.5);
\draw (9,-3)-- (7.25,-6.25);
\draw (9,-3)-- (5.5,-5.5);
\draw (5.5,1.5)-- (5.5,-5.5);
\draw (5.5,-5.5)-- (7.25,-6.25);
\draw  [line width = 6pt] (9.7,-5.5)-- (9,-3);
\draw  [line width = 6pt] (9,-7)-- (9.7,-5.5);
\fill [color=black] (4,-2) circle (6.0pt);
\draw[color=black] (3.5,-2) node {\LARGE $a$};
\fill [color=black] (9.7,-5.5) circle (6.0pt);
\draw[color=black] (10.2,-5.3) node {\LARGE $x$};
\fill [color=black] (12.5,1.5) circle (6.0pt);
\draw[color=black] (12.8,2) node {\LARGE $d$};
\fill [color=black] (14,-2) circle (6.0pt);
\draw[color=black] (14.5,-2) node {\LARGE $e$};
\fill [color=black] (9,-7) circle (6.0pt);
\draw[color=black] (9,-7.5) node {\LARGE $g$};
\fill [color=black] (9,-3) circle (6.0pt);
\draw[color=black] (9.5,-3.3) node {\LARGE $h$};
\fill [color=black] (9,3) circle (6.0pt);
\draw[color=black] (9,3.5) node {\LARGE $c$};
\fill [color=black] (5.5,1.5) circle (6.0pt);
\draw[color=black] (5.2,2) node {\LARGE $b$};
\fill [color=black] (12.5,-5.5) circle (6.0pt);
\draw[color=black] (12.8,-6) node {\LARGE $f$};
\fill [color=black] (5.5,-5.5) circle (6.0pt);
\draw[color=black] (5.2,-6) node {\LARGE $z_1$};
\fill [color=black] (7.25,-6.25) circle (6.0pt);
\draw[color=black] (7.25,-6.9) node {\LARGE $z_2$};
\end{tikzpicture}
}
\caption{A $3$-sun in $H^{+}$}\label{H+}
\end{center}
\end{figure}
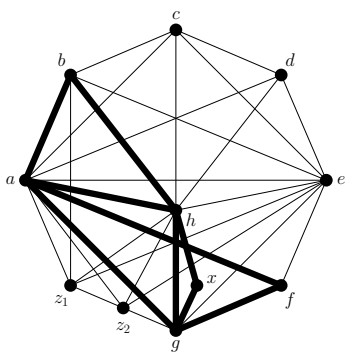
The vertices $\{a,b,f,g,h,x\}$ induce a $3$-sun in $H^{+}$.  Since $H^{+}$ is an induced subgraph of every one of the graphs constructed in Theorems \ref{sizecounter} and \ref{ratiocounter}, each such counterexample contains a sun and is thus {\em not} strongly chordal.
As such, we pose the following question:

\begin{ques}
Is every Hamiltonian strongly chordal graph fully cycle extendable?
\end{ques}

In Theorem \ref{P5}, we showed that being $P_5$-free is a sufficient condition for a Hamiltonian chordal graph to be fully cycle extendable.  However, if we take a closer look at the construction method given in Section \ref{counter}, we see that there exists a counterexample to Hendry's Conjecture that is $P_{10}$-free.  Consider again the base graph $H$ given in Figure 1.  Note that $a$ and $e$ are universal vertices, so any induced path of $H$ of length $4$ or greater contains neither vertex; Figure \ref{Hagain} shows $H$ with the edges incident to $a$ and $e$ removed.

\begin{figure}[h!]
\begin{center}
\scalebox{0.4}{
\begin{tikzpicture}
\clip(3.2,-8) rectangle (15.21,3.8);

\draw [dotted] (4,-2)-- (12.5,1.5);
\draw [dotted]  (4,-2)-- (14,-2);
\draw [dotted] (4,-2)-- (12.5,-5.5);
\draw [dotted] (4,-2)-- (9,-7);
\draw [dotted] (4,-2)-- (9,-3);
\draw [dotted] (4,-2)-- (9,3);
\draw [dotted] (4,-2)-- (5.5,1.5);
\draw [dotted] (4,-2)-- (5.5,-5.5);
\draw [dotted] (4,-2)-- (7.25,-6.25);
\draw [dotted] (12.5,1.5)-- (14,-2);
\draw [dotted]  (14,-2)-- (12.5,-5.5);
\draw [dotted] (14,-2)-- (9,-7);
\draw [dotted] (14,-2)-- (9,-3);
\draw [dotted] (14,-2)-- (9,3);
\draw [dotted] (14,-2)-- (5.5,1.5);
\draw [dotted] (5.5,-5.5)-- (14,-2);
\draw [dotted] (7.25,-6.25)-- (14,-2);

\draw (12.5,-5.5)-- (9,-7);
\draw (9,3)-- (5.5,1.5);
\draw (12.5,1.5)-- (9,-3);
\draw (12.5,1.5)-- (9,3);
\draw (9,-7)-- (9,-3);
\draw (7.25,-6.25)-- (9,-7);
\draw (9,-3)-- (9,3);
\draw (9,-3)-- (5.5,1.5);
\draw (9,-3)-- (7.25,-6.25);
\draw (9,-3)-- (5.5,-5.5);
\draw (5.5,1.5)-- (5.5,-5.5);
\draw (5.5,-5.5)-- (7.25,-6.25);
\fill [color=black] (4,-2) circle (6.0pt);
\draw[color=black] (3.5,-2) node {\LARGE $a$};
\fill [color=black] (12.5,1.5) circle (6.0pt);
\draw[color=black] (12.8,2) node {\LARGE $d$};
\fill [color=black] (14,-2) circle (6.0pt);
\draw[color=black] (14.5,-2) node {\LARGE $e$};
\fill [color=black] (9,-7) circle (6.0pt);
\draw[color=black] (9,-7.5) node {\LARGE $g$};
\fill [color=black] (9,-3) circle (6.0pt);
\draw[color=black] (9.5,-3.3) node {\LARGE $h$};
\fill [color=black] (9,3) circle (6.0pt);
\draw[color=black] (9,3.5) node {\LARGE $c$};
\fill [color=black] (5.5,1.5) circle (6.0pt);
\draw[color=black] (5.2,2) node {\LARGE $b$};
\fill [color=black] (12.5,-5.5) circle (6.0pt);
\draw[color=black] (12.8,-6) node {\LARGE $f$};
\fill [color=black] (5.5,-5.5) circle (6.0pt);
\draw[color=black] (5.2,-6) node {\LARGE $z_1$};
\fill [color=black] (7.25,-6.25) circle (6.0pt);
\draw[color=black] (7.25,-6.9) node {\LARGE $z_2$};
\end{tikzpicture}
}
\caption{$H$ without the edges incident to $a$ or $e$}\label{Hagain}
\end{center}
\end{figure}
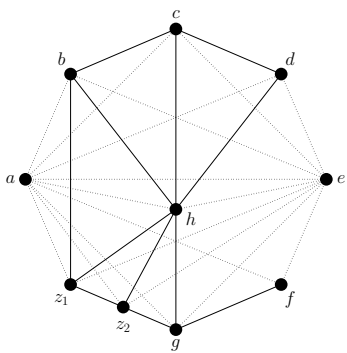

We can now see that $H - \{a,e,h\}$ is a path, and hence an induced path of $H$ on $7$ vertices.  Call this path $P$.  Let $H^*$ be the graph obtained from $H$ by pasting a triangle onto each heavy edge of $H$.  Clearly $H^*$ contains $P$ as an induced path, and each end of $P$ will be incident to a degree $2$ vertex in $H^*$.  We can thus extend $P$ to an induced path on $9$ vertices, and this is the longest such path in $H^*$.  Hence, not every $P_{10}$-free Hamiltonian chordal graph is fully cycle extendable.  We thus ask the following:

\begin{ques}\label{paths}
What is the largest integer $r$ for which every $P_r$-free Hamiltonian chordal graph is fully cycle extendable?
\end{ques}

Based on the discussion above and Theorem \ref{P5}, the answer to Question \ref{paths} is at least $5$ and at most $9$.

It is also worth noting that every counterexample given in Section \ref{counter} has a vertex cut of size $2$; this prompts the following question:

\begin{ques}
Does there exist a value $k>2$ such that every $k$-connected Hamiltonian chordal graph is fully cycle extendable?
\end{ques}

The toughness of a graph, $\tau(G)$, is the minimum value of $\frac{|X|}{c(G-X)}$ taken over all vertex cuts $X$, where $c(G-X)$ denotes the number of connected components of $G-X$.  A graph is {\em $t$-tough} if $\tau(G) \geq t$.  Note that every $t$-tough graph is $2t$-connected and that every Hamiltonian graph is $1$-tough.  As such, we immediately see that every counterexample given in Section \ref{counter} is $1$-tough, which prompts the following, more restrictive, question:

\begin{ques}
Does there exist a value $t>1$ such that every $t$-tough Hamiltonian chordal graph is fully cycle extendable?
\end{ques}

\section{Acknowledgements}

The authors thank Ge{\v n}a Hahn for the stimulating discussions.
We are also grateful for the careful reading of this paper by the referees and for their helpful comments.

\bibliographystyle{siam}
\bibliography{references}

\end{document}